\documentclass[10pt]{article}
\usepackage{amsmath,amsfonts,amsthm,epsfig, xcolor}
\usepackage{enumerate, amsmath, amsfonts, amsthm, url, bbm,
wasysym,amssymb} 
\makeatletter \@addtoreset{equation}{section}

\newcommand{\beq}[1]{\begin{equation} \label{#1}}
\newcommand{\eeq}{\end{equation}}
\newcommand{\bed}{\begin{displaymath}}
\newcommand{\eed}{\end{displaymath}}
\newcommand{\bea}{\bed\begin{array}{rl}}
\newcommand{\eea}{\end{array}\eed}
\newcommand{\barray}{\begin{array}{ll}}
\newcommand{\earray}{\end{array}}

\def\one{{\hbox{1{\kern -0.35em}1}}}

\def\qedd
{\strut\hfill\lower0.5\baselineskip\vbox{\hrule width
1em\nointerlineskip
                                        \hbox to 1em{\vrule height 1em
                                                     \hfill
                                                     \vrule height 1em}
                                        \nointerlineskip
                                        \hrule width 1em}\bigbreak}
\def\qedsym
{\vbox{\hrule width 0.5em\nointerlineskip
                                        \hbox to 0.5em{\vrule height 0.5em
                                                     \hfill
                                                     \vrule height 0.5em}
                                        \nointerlineskip
                                        \hrule width 0.5em}}

\def\({\left(}
\def\){\right)}

\newtheorem{thm}{Theorem}[section]
\newtheorem{prop}[thm]{Proposition}
\newtheorem{lem}[thm]{Lemma}

\newtheorem{rem}[thm]{Remark}

\newtheorem{exm}[thm]{Example}

\begin{document}

\title {On Complete Convergence in Mean for Double Sums of Independent Random Elements in Banach Spaces}

\author{L. V. THANH and N. T. THUY\\{\small Department
of Mathematics, Vinh University, Nghe An 42118 Vietnam}\\{\small
Email: levt@vinhuni.edu.vn, thuynt.tc3@nghean.edu.vn}}
\date{}

 \maketitle

\begin{abstract}
In this work, conditions
are provided under which a  normed double sum of independent random
elements in a real separable Rademacher type $p$
  Banach space converges completely to $0$ in mean of order $p$.
  These conditions for
the complete convergence in mean of order $p$ are shown to provide
an exact characterization of Rademacher type $p$ Banach spaces. In
case the Banach space is not of Rademacher type $p$, it is proved
that the complete convergence in mean of order $p$ of a normed
double sum implies a strong law of large numbers.



\end{abstract}

\setlength{\baselineskip}{0.25in}
\section{Introduction}\label{sec:int}

Let $\{V_{mn}, m\geq 1, n\geq 1\}$ be a double array of random
elements in a real separable Banach space $\mathcal{X}$ with norm
$\|.\|$. Throughout this paper, we write
$$S(m,n)=S_{mn}=\sum^m_{i=1}\sum^n_{j=1}V_{ij},m\ge 1,n\ge 1.$$ For
$a,b \in \Bbb{R}$, $\max\{a,b\}$ will be denoted by ${a\vee b}$. The
symbol $C$ will denote a generic constant $(0<C<\infty)$ which
 is not necessarily the same one in each appearance.

We recall that $V_{mn}$ is said to {\it converge completely to $0$}
(denoted $V_{mn}\overset{c}{\to} 0$) if
$$\sum_{m=1}^\infty\sum_{n=1}^\infty P(\|V_{mn}\|>\varepsilon)<\infty \mbox{ for all } \varepsilon>0$$
and that for $p>0$, $V_{mn}$ is said to  {\it converge to $0$ in
mean of order $p$}  as $m\vee n \to \infty$
 (denoted $V_{mn}\overset{L_p}{\to}0$ as $m\vee n \to \infty$) if
 $$E\|V_{mn}\|^p\to 0\mbox{ as }m\vee n\to\infty.$$
 By the Borel-Cantelli lemma, $V_{mn}\overset{c}{\to} 0$
 ensure that $V_{mn} \to 0$ almost surely (a.s.) as $m\vee n \to \infty$ (see, e.g., \cite{RosalskyThanh06}).
But the modes of convergence $V_{mn}\overset{c}{\to}0$ and
$V_{mn}\overset{L_p}{\to} 0$
 are not comparable in general. The double array $V_{mn}$ is said to
 {\it converges completely to $0$ in mean of order $p$} (denoted $V_{mn}\overset{c,L_p}{\to} 0$) if
  $$\sum_{m=1}^\infty\sum_{n=1}^\infty E\|V_{mn}\|^p <\infty.$$
  It is easy to see that $V_{mn}\overset{c,L_p}{\to} 0$ ensure both $V_{mn}\overset{c}{\to}
  0$ and $V_{mn}\overset{L_p}{\to}0$ as $m\vee n \to \infty$.
  However, as we
  will see later in Example \ref{exam41} that the converse is not true.

 The notion of complete convergence in mean of order $p$ ($p>0$) was
apparently first investigated by Chow \cite{Chow} in the
(real-valued) random variables case. Rosalsky, Thanh and Volodin
\cite{RTV}  studied the complete convergence in mean of order $p$
for sequences of independent random elements in Banach spaces and
provided through this mode of convergence a new characterization of
Rademacher type $p$ Banach spaces. In this paper, we establish the
double sum versions for the main results in \cite{RTV}. This is done
by using recent results by Rosalsky, Thanh and Thuy in \cite{RTT}.
 The main results are Theorems \ref{thm31} and \ref{thm32}.
 Theorem \ref{thm31} provides conditions
under which the normed double sum $S_{mn}/(mn)^{(p+1)/p}$ converges completely to $0$ in mean of order
$p$, $1 \le  p \le 2$. Moreover, these conditions for
$S_{mn}/(mn)^{(p+1)/p}$ converging
completely to $0$ in mean of order $p$ are shown to provide an exact
characterization of Rademacher type $p$ Banach spaces. Theorem
\ref{thm32}  shows that in general Banach spaces, the condition
$S_{mn}/(mn)^{(p+1)/p}\overset{c, L_p}{\to} 0 \mbox{ for some } p\ge
1$ implies the strong law of large numbers (SLLN) $S_{mn}/(mn)\to 0
\mbox{ a.s. as } m\vee n\to \infty.$

The reader may refer to Gut \cite{Gut78}, Gut and Stadtm\"{u}ller
\cite{GS10, GS11}, M\'{o}ricz \cite{Moricz80}, M\'{o}ricz, Su and
Taylor \cite{MSTa}, M\'{o}ricz, Stadtm\"{u}ller and Thalmaier
\cite{MSTh}, Smythe \cite{Smythe73} and references therein for SLLN
and other limit theorems for double arrays of random variables.
 Rosalsky and Thanh \cite{RosalskyThanh06}
gave a brief discussion of a historical nature concerning double
sums and on their importance in the field of statistical physics. In
a major surrey article \cite{Pyke}, Pyke discussed fluctuation
theory, the limiting Brownian sheet, the SLLN, and the law of the
iterated logarithm for double arrays of independent identically
distributed real-valued random variables. Recently, Klesov
\cite{Klesov14} published a comprehensive book on multiple sums.

The plan of the paper is as follows. Notation, technical
definitions, and six known propositions and lemmas which are used in
proving the main results are consolidated into Section 2. The
 main results are established in Section 3.
 In Section 4, two illustrating examples concerning the sharpness of
Theorems \ref{thm31} and \ref{thm32} are presented.

\section{Preliminaries}\label{sec:Kol}
In this section, notation, lemmas and propositions which are needed
in connection with the main results will be presented.

 The {\it expected value} or {\it mean} of a Banach space $\mathcal{X}$-valued
random element $V$, denoted $EV$, is defined to be the {\it Pettis
integral} provided it exists.  If $E\|V\|<\infty$, then (see, e.g.,
Taylor \cite[p. 40]{Taylor}) $V$ has an expected value.  But the
expected value can exist when $E\|V\|=\infty$.  For an example, see
Taylor \cite[p. 41]{Taylor}.

The reader may refer to Hoffmann-J{\o}rgensen and Pisier
\cite{Hoffmann} for definition, properties and examples of
Rademacher type $p$ Banach spaces. Hoffmann-J{\o}rgensen and Pisier
\cite{Hoffmann} proved for $1\leq p\leq 2$ that a real separable
Banach space is of Rademacher type $p$ if and only if there exists a
constant $C$ depending only on $p$ such that

\begin{equation}
  \label{HJP}
  E\Bigg|\Bigg|\sum^n_{j=1}V_j\Bigg|\Bigg|^p \leq C\sum^n_{j=1} E||V_j||^p
\end{equation}
for every finite collection $\{V_1,\dots,V_n\}$ of independent mean
0 random elements.

The proof of the following simple lemma can be found in \cite{RTV}.

\begin{lem}\label{RTV}
Let $\{V_n, n\ge 1\}$ be a sequence of independent mean $0$ random
elements in a real separable Banach space. Then for all $ p\ge 1$,
the sequence $ \{E\|\sum_{j=1}^n V_{j}\|^p,n\ge 1\}$ is
nondecreasing.
\end{lem}

Proposition \ref{RT07} is a double sum analogue of the classical Kolmogorov
SLLN in Banach spaces.

\begin{prop}[Rosalsky and Thanh \cite{RosalskyThanh07}]\label{RT07}
Let $1\le p \le 2$ and let $\mathcal{X}$ be real separable Banach
space. Then the following two statements are equivalent:

 (i) The Banach space $\mathcal{X}$ is of Rademacher type $p$.

 (ii) For every double array $\{V_{mn},m\geq 1,n\geq 1\}$ of independent
mean 0 random elements in $\mathcal{X}$ and every choice of
constants $\alpha > 0$ and $\beta > 0$, the condition
\begin{equation}
\label{Kol1}\sum_{m=1}^\infty\sum_{n=1}^\infty
\dfrac{E\|V_{mn}\|^p}{m^{\alpha p}n^{\beta p}}<\infty
\end{equation}
implies that the SLLN
\begin{equation}\label{Kol2}
\dfrac{S_{mn}}{m^{\alpha}n^{\beta}} \to 0 \mbox{ a.s. as } m\vee
n\to \infty
\end{equation}
obtains.
\end{prop}

If the Banach space is not of Rademacher type $p$, condition
\eqref{Kol1} alone does not ensure the SLLN \eqref{Kol2} (see \cite[Example 5.1]{RTT}). The next
proposition is a recent result of Rosalsky, Thanh and Thuy
\cite{RTT} which considers the law of large numbers for double sums
in a general real separable Banach space. It shows that if
\eqref{Kol1} holds, the the SLLN \eqref{Kol2} and the weak law of large numbers (WLLN)
\eqref{RTT01} are equivalent.
\begin{prop}[Rosalsky, Thanh and Thuy \cite{RTT}]\label{RTT}
Let $\alpha>0,\beta>0$ and let $\{V_{mn},m\geq 1,n\geq 1\}$ be a
double array of independent random elements in a real separable
Banach space. Assume that \eqref{Kol1} holds for some $1\le p\le 2$,
then the SLLN \eqref{Kol2} holds if and only if
 \begin{equation}\label{RTT01}\dfrac{S_{mn}}{m^\alpha n^\beta} \overset{P}{\to} 0\
\mbox{as}\ m\vee n\to \infty.
\end{equation}
\end{prop}

The next lemma is Lemma 3.2 in \cite{RTT} which enables to study the
SLLN through the symmetrization procedures.
\begin{lem}[Rosalsky, Thanh and Thuy \cite{RTT}]\label{RTT32}
Let $\alpha>0,\beta>0$ and let $V=\{V_{mn},m\geq 1,n\geq1\}$ and
$V'=\{V_{mn}^{'},m\geq 1,n\geq 1\}$ be two double arrays of
independent random elements in a real separable Banach space such
that
 $V$ and $V'$ are independent copies of each other. Let $S^*_{mn}=\sum_{i=1}^m\sum_{j=1}^n(V_{ij}-V'_{ij})$. Then
 $$\dfrac{S_{mn}}{m^\alpha n^\beta}\to 0\
\mbox{a.s. as}\ m\vee n\to \infty$$ if and only if
 $$\dfrac{S^*_{mn}}{m^\alpha n^\beta}\to 0\ \mbox{a.s. as}
\ m\vee n\to \infty$$
 and $$\dfrac{S_{mn}}{m^\alpha n^\beta}
\overset{P}{\to}0\ \mbox{ as }\ m\vee n\to \infty.$$
\end{lem}

Lemma \ref{RTT34} considers the SLLN for double arrays of symmetric
independent random elements.
\begin{lem}[Rosalsky, Thanh and Thuy \cite{RTT}]\label{RTT34}
Let $\{V_{mn},m\geq 1,n\geq1\}$ be a double array of independent
symmetric random elements in a real separable Banach space. Then
$$\dfrac{S_{mn}}{mn}\to 0 \mbox{ a.s. as } m\vee n\to\infty$$
 if
and only if
$$\dfrac{\sum_{i=2^m+1}^{2^{m+1}}\sum_{j=2^n+1}^{2^{n+1}}V_{ij}}{2^m2^n}\to
0\mbox{ a.s. as } m\vee n\to\infty.$$
\end{lem}

The last lemma is a simple consequence of Lemma 1 of Etemadi
\cite{Etemadi85b}.
\begin{lem}\label{Ete85} Let $X$ and $Y$ be two independent symmetric
random elements in a real separable Banach space. Then for all
$t>0$,
\begin{equation*}P(\|X\|>t)\leq 2P(\|X+Y\|>t).
\end{equation*}
\end{lem}

\section{Main Results}
With the preliminaries accounted for, the first main result may be
established. Theorem \ref{thm31} provides a new characterization of
Rademacher type $p$ Banach spaces through the complete convergence
in mean of order $p$ for normed double sums.

\begin{thm}\label{thm31}
Let $ 1\leq p\leq 2$ and let $\mathcal{X}$ be a real separable
Banach space. Then the following statements are equivalent:
\begin{description}
\item{(i)} $\mathcal{X}$ is of Rademacher type $p$.

\item{(ii)}  For every double array $\{V_{mn},m\geq 1,n\geq 1\}$ of independent
 mean 0 random elements in $\mathcal{X}$, the condition
 \begin{equation}\label{31}\sum_{m=1}^\infty\sum_{n=1}^\infty
 \dfrac{E\|V_{mn}\|^p}{m^pn^p}<\infty
 \end{equation}
 implies
 \begin{equation}\label{32}\dfrac{S_{mn}}{(mn)^{(p+1)/p}}
 \overset{c,L_p}{\to}0.
 \end{equation}
 \end{description}
\end{thm}

\begin{proof}
 Note that
 \begin{equation}\label{33}\sum_{m=i}^\infty\sum_{n=j}^\infty \dfrac{1}{(mn)^{p+1}}\approx
 \dfrac{1}{p^2}\dfrac{1}{(ij)^p}.
 \end{equation}
 Assume that (i) holds. Let $\{V_{mn},m\geq 1,n\geq 1\}$ be a
 double array
  of independent mean $0$ random elements in $\mathcal{X}$ satisfying \eqref{31}. Then

\begin{align*}\sum_{m=1}^\infty\sum_{n=1}^\infty E\Big\| \dfrac{S_{mn}}{(mn)^{(p+1)/p}}\Big\|^p
&\leq C\sum_{m=1}^\infty\sum_{n=1}^\infty \dfrac{\sum_{i=1}^m\sum_{j=1}^n E\|V_{ij}\|^p}{(mn)^{p+1}}\mbox{ (by \eqref{HJP})}\\
& =C\sum_{i=1}^\infty\sum_{j=1}^\infty\sum_{m=i}^\infty\sum_{n=j}^\infty\dfrac{E\|V_{ij}\|^p}{(mn)^{p+1}}\\
&=C\sum_{i=1}^\infty\sum_{j=1}^\infty E\|V_{ij}\|^p\sum_{m=i}^\infty\sum_{n=j}^\infty \dfrac{1}{(mn)^{p+1}}\\
& \leq C \sum_{i=1}^\infty\sum_{j=1}^\infty
\dfrac{1}{p^2}\dfrac{E\|V_{ij}\|^p}{(ij)^p}\mbox{ (by
\eqref{33})}\\
&<\infty \mbox{ (by \eqref{31})}.
\end{align*}
 So \eqref{32} is proved. This ends the proof of the implication ((i)$\Rightarrow$(ii)).

 Now, assume that (ii) holds. Let $\{V_{mn},m\geq 1,n\geq1\}$ be a double array of
    independent mean $0$ random elements in $\mathcal{X}$ such that \eqref{31} holds.
    In view of Proposition \ref{RT07}, it suffices to verify that
\begin{equation}\label{34}\dfrac{S_{mn}}{mn}\to 0 \mbox{ a.s. as }m\vee n\to
\infty.
\end{equation}
Now \eqref{32} holds by \eqref{31} and (ii) and so
\begin{equation}\label{35}\sum_{m=1}^\infty \sum_{n=1}^\infty \dfrac{E\|S_{mn}\|^p}{(mn)^{p+1}}
<\infty.
\end{equation}
Thus
\begin{align*}
E\left\|\dfrac{S_{mn}}{mn}\right\|^p&=\dfrac{1}{(mn)^p}E\|S_{mn}\|^p\\
&\leq C \sum_{k=m}^\infty\sum_{l=n}^\infty \dfrac{1}{(kl)^{p+1}}E\|S_{mn}\|^p \mbox{ (by\eqref{33}) }\\
& \leq C\sum_{k=m}^\infty\sum_{l=n}^\infty
\dfrac{1}{(kl)^{p+1}}E\|S_{kl}\|^p \mbox{ (by Lemma \ref{RTV})}\\
&\to 0 \mbox{ as } m\vee n\to \infty \mbox{ (by \eqref{35})}.
\end{align*}
Then by Markov's inequality $\dfrac{S_{mn}}{mn}\overset{P}{\to} 0$
as $m\vee n\to\infty$ and so \eqref{34} holds by Proposition
\ref{RTT}. The proof of the implication ((ii)$\Rightarrow$(i)) is
completed.
\end{proof}

\begin{rem}\label{rm1}{\rm
From the proof of Theorem \ref{thm31}, we see that if \eqref{35}
holds for some $p\ge 1$, we obtain
$$\dfrac{S_{mn}}{mn}{\to }\ 0\mbox{ in $L_p$ as }m\vee n\to\infty.$$ This remark will be used in the proof of Theorem \ref{thm32}.}
\end{rem}

In the following theorem, we show that
$S_{mn}/(mn)^{(p+1)/p}\overset{c, L_p}{\to} 0 \mbox{ for some } p\ge
1$ implies $S_{mn}/(mn)\to 0 \mbox{ a.s. as } m\vee n\to
\infty.$ We emphasize that we are not assuming that the Banach space
is of Rademacher type $p$.

\begin{thm}\label{thm32}
Let $\{V_{mn},m\geq 1,n\geq 1\}$ be a double array of independent
random elements in a real separable Banach space. If
\begin{equation}\label{38} \dfrac{S_{mn}}{(mn)^{(p+1)/p}}\overset{c, L_p}{\to} 0
\mbox{ for some } p\ge 1,
\end{equation}
then
\begin{equation}\label{39}\dfrac{S_{mn}}{mn}\to 0 \mbox{ a.s. as } m\vee n\to
\infty.
\end{equation}
\end{thm}

\begin{rem}\label{rm2} {\rm
\begin{description}
\item{(i)}  In \cite{RTV},
Rosalsky, Thanh and Volodin established Theorem \ref{thm32} for
$1$-dimensional case with $1\le p\le 2$. The proof we presented here for
the double sum version is much more complicated.
As we will see in the proof
that the condition $p\le 2$ is not needed.

\item{(ii)} Recently, Son, Thang and Dung \cite{STD} proved a result on
complete convergence in mean of order $p$ without assuming that the
summands are independent. More precise, they proved that for arbitrary
double array $\{V_{mn},m\ge 1,n\ge 1\}$ in a real separable Banach
space, the condition
$$\dfrac{1}{(mn)^{(p+1)/p}}\max_{k\le m,l\le n}\|S_{kl}\|\overset{c, L_p}{\to} 0
\mbox{ for some } 1\le p\le 2$$ implies
$$\dfrac{1}{mn}\max_{k\le m,l\le n}\|S_{kl}\|\to 0 \mbox{ a.s. as } m\vee n\to
\infty.$$ Their result and ours are not comparable and do not imply
each other, and our proof is completely different from theirs.
Moreover, we will show in Example \ref{exm42} that in our Theorem
\ref{thm32}, the independence assumption cannot be weakened to the
assumption that the random elements are pairwise independent.
\end{description}
 }
\end{rem}

 The proof of Theorem \ref{thm32} has several
steps so we will break it up into two lemmas.  These lemmas may be
of independent interest. The first lemma provides a necessary and
sufficient condition for SLLN $S_{mn}/(mn)\to 0 \mbox{ a.s. as }
m\vee n\to \infty$ when $\{V_{mn},m\geq 1, n\geq 1\}$  is comprised
of independent symmetric random elements. Lemma \ref{lem33} is a
double sum analogue of Theorem 1 of Etemadi \cite{Etemadi85a}.

\begin{lem}\label{lem33}Let $\{V_{mn},m\geq 1,n\geq 1\}$ be a double array of independent
symmetric random elements in a real separable Banach space. Then
\begin{equation}\label{36}\dfrac{S_{mn}}{mn}\to 0 \mbox{ a.s. as } m\vee n \to \infty
\end{equation}
if and only if
\begin{equation}\label{37}\sum_{m=1}^\infty\sum_{n=1}^\infty \dfrac{1}{mn}P\left(\left\|\sum_{i=m+1}^{2m}\sum_{j=n+1}^{2n}V_{ij}\right\|
> \varepsilon mn\right)<\infty \mbox{ for all }\varepsilon>0.
\end{equation}
\end{lem}

\begin{proof}
Assume that \eqref{36} holds and let $\varepsilon>0$ be arbitray. It
is easy to see that \eqref{36} implies
\begin{equation*}
\dfrac{\sum_{i=2^k+1}^{2^{k+1}}\sum_{j=2^{l}+1}^{2^{l+1}}V_{ij}}{2^k2^l}\to
0\mbox{ a.s. as } k\vee l\to\infty.
\end{equation*}
By the Borel-Cantelli lemma, it implies
\begin{equation}\label{3a}
\sum_{k=0}^\infty\sum_{l=0}^\infty
P\left(\left\|\sum_{i=2^k+1}^{2^{k+1}}\sum_{j=2^{l}+1}^{2^{l+1}}V_{ij}
\right\|>\dfrac{\varepsilon}{4} 2^k2^{l}\right)<\infty.
\end{equation}
Similarly, we have
\begin{equation}\label{3b}
\sum_{k=0}^\infty\sum_{l=0}^\infty
P\left(\left\|\sum_{i=2^{k+1}+1}^{2^{k+2}}\sum_{j=2^{l}+1}^{2^{l+1}}V_{ij}
\right\|>\dfrac{\varepsilon}{4} 2^k2^{l}\right)<\infty,
\end{equation}
\begin{equation}\label{3c}
\sum_{k=0}^\infty\sum_{l=0}^\infty
P\left(\left\|\sum_{i=2^k+1}^{2^{k+1}}\sum_{j=2^{l+1}+1}^{2^{l+2}}V_{ij}
\right\|>\dfrac{\varepsilon}{4} 2^k2^{l}\right)<\infty,
\end{equation}
and
\begin{equation}\label{3d}
\sum_{k=0}^\infty\sum_{l=0}^\infty
P\left(\left\|\sum_{i=2^{k+1}+1}^{2^{k+2}}\sum_{j=2^{l+1}+1}^{2^{l+2}}V_{ij}
\right\|>\dfrac{\varepsilon}{4} 2^k2^{l}\right)<\infty.
\end{equation}
Now, we have
\begin{align*}&\sum_{m=2}^\infty\sum_{n=2}^\infty \dfrac{1}{mn}P\left(\left\|\sum_{i=m+1}^{2m}\sum_{j=n+1}^{2n}V_{ij}\right\|>\varepsilon
mn\right)\\
&\quad\quad=\sum_{k=0}^\infty\sum_{l=0}^\infty\sum_{m=2^k+1}^{2^{k+1}}\sum_{n=2^l+1}^{2^{l+1}}
\dfrac{1}{mn}P\left(\left\|\sum_{i=m+1}^{2m}\sum_{j=n+1}^{2n}V_{ij}\right\|>\varepsilon mn\right)\\
&\quad\quad \leq 2\sum_{k=0}^\infty\sum_{l=0}^\infty
\dfrac{1}{2^k2^l}\sum_{m=2^k+1}^{2^{k+1}}\sum_{n=2^l+1}^{2^{l+1}}
P\left(\left\|\sum_{i=2^k+1}^{2^{k+2}}\sum_{j=2^l+1}^{2^{l+2}}V_{ij}\right\|>\varepsilon 2^k2^l\right)\mbox{ (by Lemma \ref{Ete85})}\\
&\quad\quad = 2\sum_{k=0}^\infty\sum_{l=0}^\infty
P\left(\left\|\sum_{i=2^k+1}^{2^{k+2}}\sum_{j=2^l+1}^{2^{l+2}}
V_{ij}\right\|>\varepsilon 2^k 2^l\right)\\
&\quad\quad \leq 2\sum_{k=0}^\infty\sum_{l=0}^\infty
P\left(\left\|\sum_{i=2^k+1}^{2^{k+1}}\sum_{j=2^l+1}^{2^{l+1}}V_{ij}\right\|>
\dfrac{\varepsilon}{4}2^k2^l\right)\\
&\quad\quad\quad\quad\quad + 2\sum_{k=0}^\infty\sum_{l=0}^\infty
P\left(\left\|\sum_{i=2^{k+1}+1}^{2^{k+2}}\sum_{j=2^{l}+1}^{2^{l+1}}V_{ij} \right\|>\dfrac{\varepsilon}{4}2^{k}2^{l}\right)\\
&\quad\quad\quad\quad\quad + 2\sum_{k=0}^\infty\sum_{l=0}^\infty
P\left(\left\|\sum_{i=2^k+1}^{2^{k+1}}\sum_{j=2^{l+1}+1}^{2^{l+2}}V_{ij}\right\|>\dfrac{\varepsilon}{4}2^k2^{l}\right)\\
&\quad\quad\quad\quad\quad + 2\sum_{k=0}^\infty\sum_{l=0}^\infty
P\left(\left\|\sum_{i=2^{k+1}+1}^{2^{k+2}}\sum_{j=2^{l+1}+1}^{2^{l+2}}V_{ij}\right\|>\dfrac{\varepsilon}{4}2^{k}2^{l}\right)\\
&\quad\quad<\infty \mbox{ (by \eqref{3a}-\eqref{3d}).}
\end{align*}
The proof of the implication (\eqref{36}$\Rightarrow$\eqref{37}) is
thus completed.

Now, we assume \eqref{37} holds. Then for arbitrary $\varepsilon>0$,
\begin{align*}
\infty&>\sum_{m=4}^\infty\sum_{n=4}^\infty\dfrac{1}{mn}P\left(\left\|\sum_{i=m+1}^{2m}\sum_{j=n+1}^{2n}V_{ij}\right\|
>\varepsilon mn\right)\\
&=\sum_{k=1}^\infty\sum_{l=1}^\infty
\sum_{m=2^k+2^{k-1}+1}^{2^k+2^{k+1}}\sum_{n=2^l+2^{l-1}+1}^{2^l+2^{l+1}}
\dfrac{1}{mn}P\left(\left\|\sum_{i=m+1}^{2m}\sum_{j=n+1}^{2n}V_{ij}\right\|> \varepsilon mn\right)\\
&
=\sum_{k=1}^\infty\sum_{l=1}^\infty\sum_{m=2^k+2^{k-1}+1}^{2^{k+1}}\sum_{n=2^l+2^{l-1}+1}^{2^{l+1}}
\dfrac{1}{mn}P\left(\left\|\sum_{i=m+1}^{2m}\sum_{j=n+1}^{2n}V_{ij}\right\|>\varepsilon mn\right)\\
&\quad\quad +\sum_{k=1}^\infty\sum_{l=1}^\infty
\sum_{m=2^{k+1}+1}^{2^{k+1}+2^k}\sum_{n=2^l+2^{l-1}+1}^{2^{l+1}}
\dfrac{1}{mn}P\left(\left\|\sum_{i=m+1}^{2m}\sum_{j=n+1}^{2n}V_{ij}\right\|>\varepsilon mn\right)\\
&\quad\quad+\sum_{k=1}^\infty\sum_{l=1}^\infty\sum_{m=2^k+2^{k-1}+1}^{2^{k+1}}\sum_{n=2^{l+1}+1}^{2^{l+1}+2^l}
\dfrac{1}{mn}P\left(\left\|\sum_{i=m+1}^{2m}\sum_{j=n+1}^{2n}V_{ij}\right\|>\varepsilon mn\right)\\
&\quad\quad+\sum_{k=1}^\infty\sum_{l=1}^\infty\sum_{m=2^{k+1}+1}^{2^{k+1}+2^k}\sum_{n=2^{l+1}+1}^{2^{l+1}+2^l}
\dfrac{1}{mn}P\left(\left\|\sum_{i=m+1}^{2m}\sum_{j=n+1}^{2n}V_{ij}\right\|>\varepsilon
mn\right)\\
& \geq \dfrac{1}{32}\sum_{k=1}^\infty\sum_{l=1}^\infty
P\left(\left\|\sum_{i=2^{k+1}+1}^{2^{k+1}+2^k}\sum_{j=2^{l+1}+1}^{2^{l+1}+2^l}V_{ij}\right\|>\varepsilon 2^{k+2}2^{l+2}\right)\\
&\quad\quad+\dfrac{1}{24}\sum_{k=1}^\infty\sum_{l=1}^\infty
P\left(\left\|\sum_{i=2^{k+1}+2^k+1}^{2^{k+2}}\sum_{j=2^{l+1}+1}^{2^{l+1}+2^l}V_{ij}\right\|>\varepsilon 2^{k+2}2^{l+2}\right)\\
&\quad\quad+\dfrac{1}{24}\sum_{k=1}^\infty\sum_{l=1}^\infty
P\left(\left\|\sum_{i=2^{k+1}+1}^{2^{k+1}+2^k}\sum_{j=2^{l+1}+2^l+1}^{2^{l+2}}V_{ij}\right\|>\varepsilon 2^{k+2}2^{l+2}\right)\\
&\quad\quad+\dfrac{1}{18}\sum_{k=1}^\infty\sum_{l=1}^\infty
P\left(\left\|\sum_{i=2^{k+1}+2^k+1}^{2^{k+2}}\sum_{j=2^{l+1}+2^l+1}^{2^{l+2}}V_{ij}\right\|>\varepsilon 2^{k+2}2^{l+2}\right)\\
&\quad\quad\quad \quad\quad\quad \quad\quad\quad \mbox{ (by Lemma \ref{Ete85})}\\
& \geq \dfrac{1}{32}\sum_{k=1}^\infty\sum_{l=1}^\infty
P\left(\left\|\sum_{i=2^{k+1}+1}^{2^{k+2}}\sum_{j=2^{l+1}+1}^{2^{l+2}}V_{ij}\right\|>4\varepsilon
2^{k+2}2^{l+2}\right).
\end{align*}
By the Borel-Cantelli lemma, it follows that
\begin{equation}\label{3e}
\dfrac{\sum_{i=2^{k+1}+1}^{2^{k+2}}\sum_{j=2^{l+1}+1}^{2^{l+2}}V_{ij}}{2^{k+2}2^{l+2}}\to
0 \mbox{ a.s. as }k\vee l\to\infty. \end{equation} Applying Lemma
\ref{RTT34}, \eqref{36} follows from \eqref{3e}.
\end{proof}

The following lemma is similar to Lemma \ref{lem33} but the random
elements $\{V_{mn},m\ge1,n\ge1\}$ are not assumed to be symmetric.
It is a double sum analogue of Theorem 2 of Etemadi
\cite{Etemadi85a}.
\begin{lem}\label{lem34}
Let $\{V_{mn},m\geq 1,n\geq1\}$ be a double array of independent
random elements in a real separable Banach space. Then
\begin{equation}\label{21}\dfrac{S_{mn}}{mn} \to 0 \mbox{ a.s. as } m\vee n\to \infty
\end{equation}
if and only if
\begin{equation}\label{22a}\dfrac{S_{mn}}{mn}\overset{P}{\to} 0 \mbox{ a.s. as } m\vee n\to \infty
\end{equation}
and
\begin{equation}\label{22b}\sum_{m=1}^\infty \sum_{n=1}^\infty
\dfrac{1}{mn}P\left(\left\|\sum_{i=m+1}^{2m}\sum_{j=n+1}^{2n}V_{ij}\right\|>\varepsilon
mn\right)<\infty \mbox{ for all }\varepsilon>0.
\end{equation}
\end{lem}
\begin{proof}
Let $V'=\{V_{mn}^{'},m\geq 1,n\geq 1\}$ and $ S^*_{mn}$ be as in
Lemma \ref{RTT32} and set
 $$Y_{mn}=\sum_{i=m+1}^{2m}\sum_{j=n+1}^{2n}V_{ij},\ m\ge 1, n\ge 1,$$
  and
$$Y'_{mn}=\sum_{i=m+1}^{2m}\sum_{j=n+1}^{2n}V'_{ij},\ m\ge 1, n\ge
1.$$

{\it Proof of the implication} (\eqref{21}$\Rightarrow$ \eqref{22a}
{\it and } \eqref{22b}): Assume \eqref{21} holds and let
$\varepsilon>0$ be arbitrary. By using Lemma \ref{RTT32}, we get
 \begin{equation*}
 \dfrac{\sum_{i=1}^m\sum_{j=1}^n(V_{ij}-V'_{ij})}{mn}\to 0 \mbox{ a.s. as } m\vee n\to
 \infty.
 \end{equation*}
It follows from Lemma \ref{lem33} that
 \begin{equation}\label{382}\sum_{m=1}^\infty\sum_{n=1}^\infty\dfrac{1}{mn}P\left(\left\|\sum_{i=m+1}^{2m}\sum_{j=n+1}^{2n}(V_{ij}-V'_{ij})\right\|
>\dfrac{\varepsilon}{2} mn\right)<\infty.
 \end{equation}
Let $\mu_{mn}=\mbox{ median of } \|Y_{mn}\|$. Then it is clear that
\eqref{21} implies
$$\dfrac{\mu_{mn}}{mn}\to 0 \mbox{
 as } m\vee n\to \infty.$$
 Thus, for $k\vee l$ large enough,
\begin{align*}\sum_{m=k}^\infty\sum_{n=l}^\infty\dfrac{1}{mn}P\left(\left\|Y_{mn}\right\|>
 \varepsilon mn\right)
 & \leq
 \sum_{m=k}^\infty\sum_{n=l}^\infty\dfrac{1}{mn}P\left(\left|\|Y_{mn}\|-\mu_{mn}\right|>
 \dfrac{\varepsilon }{2}mn\right)\\
 & \leq 2\sum_{m=k}^\infty\sum_{n=l}^\infty\dfrac{1}{mn}P\left(\left|\|Y_{mn}\|-\|Y'_{mn}\|\right|>
 \dfrac{\varepsilon}{2}mn\right)\\
 &\mbox{ (by the weak symmetrization inequality \cite[p.134]{Gut13})}\\
  &\leq 2\sum_{m=k}^\infty\sum_{n=l}^\infty\dfrac{1}{mn}P\left(\left\|Y_{mn}-Y'_{mn}\right\|>
 \dfrac{\varepsilon}{2}mn\right)\\
 &<\infty \mbox{ (by \eqref{382})}
 \end{align*}
 thereby proving \eqref{22b}. Of course \eqref{21} immediately
 implies \eqref{22a}.

{\it Proof of the implication} (\eqref{22a} {\it and } \eqref{22b}
$\Rightarrow$ \eqref{21}): Assume that \eqref{22a} and \eqref{22b}
 hold.
Again, let $\varepsilon>0$ be arbitrary, then
\begin{align*}\
\sum_{m=1}^\infty\sum_{n=1}^\infty
\dfrac{1}{mn}P(\|Y_{mn}-Y'_{mn}\|>\varepsilon mn)&\leq 2
\sum_{m=1}^\infty\sum_{n=1}^\infty
\dfrac{1}{mn}P\left(\left\|Y_{mn}\right\|>\dfrac{\varepsilon}{2}mn\right)\\
&<\infty \mbox{ (by \eqref{22b}).}
\end{align*}
It thus follows from Lemma \ref{lem33} that
\begin{equation}\label{22c}
\dfrac{S_{mn}^{*}}{mn}\to 0 \mbox{ a.s. as } m\vee n\to\infty.
\end{equation}
By applying Lemma \ref{RTT32}, \eqref{21} follows from \eqref{22a}
and \eqref{22c}.
\end{proof}

\begin{proof}[Proof of Theorem \ref{thm32}] From \eqref{38}, we have
\begin{equation}\label{22d}\sum_{m=1}^\infty\sum_{n=1}^\infty\dfrac{E\left\|S_{mn}\right\|^p}{(mn)^{p+1}}<\infty
\end{equation}
Using Remark \ref{rm1}, we get from \eqref{22d} that
\begin{equation}\label{22f}\dfrac{S_{mn}}{mn}{\to}\ 0 \mbox{ in $L_p$ as }
m\vee n\to\infty.
\end{equation}
It thus follows from \eqref{22f} and Markov's inequality that
\begin{equation}\label{22e}\dfrac{S_{mn}}{mn}\overset{P}{\to}0 \mbox{ as }
m\vee n\to\infty.
\end{equation}
On the other hand,
\begin{align*}& \sum_{m=1}^\infty \sum_{n=1}^\infty \dfrac{1}{mn}P\left(\left\|\sum_{i=m+1}^{2m}\sum_{j=n+1}^{2n}V_{ij}\right\|>
\varepsilon mn\right)\\
&\quad\leq \sum_{m=1}^\infty\sum_{n=1}^\infty
\dfrac{1}{\varepsilon^p(mn)^{p+1}}
E\left\|\sum_{i=m+1}^{2m}\sum_{j=n+1}^{2n}V_{ij}\right\|^p \mbox{ (by Markov's inequality)}\\
&\quad=\sum_{m=1}^\infty\sum_{n=1}^\infty \dfrac{1}{\varepsilon^p(mn)^{p+1}}E\left\|S_{(2m,2n)}-S_{(2m,n)}-S_{(m,2n)}+S_{(m,n)}\right\|^p\\
& \quad\leq \sum_{m=1}^\infty\sum_{n=1}^\infty \dfrac{C}{(mn)^{p+1}}
E\left (\left\|S_{(2m,2n)}\right\|^p+\left\|S_{(2m,n)}\right\|^p+\left\|S_{(m,2n)}\right\|^p+\left\|S_{(m,n)}\right\|^p\right)\\
&\quad <\infty \mbox{ (by \eqref{22d}).}
\end{align*}
Combining this and \eqref{22e}, we see that the conclusion
\eqref{39} follows from Lemma \ref{lem34}.
\end{proof}

\section{Illustrating Examples}
By Theorem \ref{thm31}, if a real separable Banach space is not of
Rademacher type $p$ where $1<p\le 2$, then there exists a double
array of independent mean $0$ random elements for which \eqref{31}
holds but \eqref{32} fails. The following example, which was
inspired by an example of Rosalsky and Thanh \cite{RosalskyThanh06},
exhibits such a double array of random elements in the Banach space
$\ell_1$. This example will also demonstrate that, there exists a
double array of random elements $\{T_{mn},m\ge 1,n\ge 1\}$
satisfying $ T_{mn}\overset{c}{\to}0$ and $T_{mn}\overset{L_p}{\to}
0$ as $m\vee n\to\infty$, but $ T_{mn}\overset{c,L_p}{\nrightarrow}
0$.

\begin{exm}\label{exam41}{\rm
 Let $1<p\leq 2$ and consider the Banach space
$\ell_1$ (which is not of Rademacher type $p$). Let $v^{(k)}$ denote
 the element of $\ell_1$ having $1$ in its $k^{\text{th}}$ position and $0$
 elsewhere, $k\geq 1$. Let $\varphi : \mathbb{N} \times \mathbb{N} \to \mathbb{N}$ be a one-to-one and
 onto mapping. Let $\{V_{mn}, m\geq 1, n\geq 1\}$ be a double array of independent
 random elements in $\ell_1$ by requiring the $\{V_{mn}, m\geq 1, n\geq 1\}$
 to be independent with
 $$ P\left(V_{mn}=v^{(\varphi(m,n))}\right)=
 P\left(V_{mn}=-v^{(\varphi(m,n))}\right)=\dfrac{1}{2},\ m\geq 1,n\geq 1.$$
 We have
 $$ \sum_{m=1}^{\infty}\sum_{n=1}^{\infty}\dfrac{E\left\|V_{mn}\right\|^p}{(mn)^p}=
 \sum_{m=1}^{\infty}\sum_{n=1}^{\infty}\dfrac{1}{(mn)^p}<\infty.$$
 Hence \eqref{31} holds but
 \begin{equation}\label{41}
 \sum_{m=1}^{\infty}\sum_{n=1}^{\infty}E\left\|\dfrac{S_{mn}}{(mn)^{(p+1)/p}}\right\|^p=
 \sum_{m=1}^{\infty}\sum_{n=1}^{\infty}\dfrac{1}{mn}=\infty
 \end{equation}
 and so \eqref{32} fails. Moreover, since for all $\varepsilon>0$
 and all large $ m\vee n$
 $$P\left(\dfrac{\left\|S_{mn}\right\|}{(mn)^{(p+1)/p}}>\varepsilon\right)=P\left(\dfrac{1}{(mn)^{1/p}}>\varepsilon\right)=0,$$
 it follows that
  $$\dfrac{S_{mn}}{(mn)^{(p+1)/p}}\overset{c}{\to}
  0.$$  Now by the computation in \eqref{41}, we have
  $$ E\left\|\dfrac{S_{mn}}{(mn)^{(p+1)/p}}\right\|^p=
  \dfrac{1}{mn}\to 0\mbox{ as } m\vee n\to\infty$$
  and so $$ \dfrac{S_{mn}}{(mn)^{(p+1)/p}}
  \overset{L_p} {\to} 0 \mbox{ as } m\vee n\to\infty.$$
  Consequently
  $$\dfrac{S_{mn}}{(mn)^{(p+1)/p}}\overset{c}{\to}0 \mbox{ and }
  \dfrac{S_{mn}}{(mn)^{(p+1)/p}}\overset{L_p}{\to}0 \mbox{ as } m\vee n\to\infty$$
  but  $$\dfrac{S_{mn}}{(mn)^{(p+1)/p}}\overset{c,L_p}{\nrightarrow}0.$$}
 \end{exm}

The following example shows that in general, the independence
assumption in Theorem \ref{thm32} cannot be weakened to the
assumption that the summands are pairwise independent. The example is based
on Theorem 3 in Cs\"{o}rgo, Tandori and Totik \cite{CTT}.

\begin{exm}\label{exm42}{\rm
Cs\"{o}rgo, Tandori and Totik \cite[Theorem 3]{CTT} constructed a
sequence of pairwise independent real-valued random variables
$\{X_m,m\ge1\}$ satisfying $EX_m=0,EX_{m}^2<\infty$, and
\begin{equation}\label{CTT1} \sum_{m=2}^\infty \dfrac{EX_{m}^2
(\log(\log m))^{1-\varepsilon}}{m^2}<\infty,\ \varepsilon>0,
\end{equation}
\begin{equation}\label{CTT2} P\left(\limsup_{m\to\infty}\dfrac{\left|\sum_{i=1}^m
X_i\right|}{m}=\infty\right)>0.
\end{equation}
For $m\ge 1$ we set $V_{mn}=X_{m}$ if $n=1$ and $V_{mn}=0$ if
$n\ge2$. In Theorem \ref{thm32}, let $p=2$, then
\begin{align*}\sum_{m=1}^\infty\sum_{n=1}^\infty E\Big\| \dfrac{S_{mn}}{(mn)^{(p+1)/p}}\Big\|^p
&= \sum_{m=1}^\infty\sum_{n=1}^\infty \dfrac{\sum_{i=1}^m\sum_{j=1}^n EV_{ij}^2}{(mn)^{3}}\\
&= \sum_{m=1}^\infty\dfrac{\sum_{i=1}^m EX_{i}^2}{m^{3}}\\
&= \sum_{i=1}^\infty\sum_{m=i}^\infty\dfrac{EX_{i}^2}{m^{3}}\\
&\le C\sum_{i=1}^\infty \dfrac{EX_{i}^2}{i^2}<\infty \mbox{ (by
\eqref{CTT1})}.
\end{align*}
 So \eqref{38} holds. However, it follows from \eqref{CTT2} that
 \eqref{39} fails.

 }
\end{exm}

\begin{center}
{\bf Acknowledgements}
\end{center}
The research was supported by the Vietnam Institute for Advanced
Study in Mathematics (VIASM) and Ministry of Education and Training,
grant no. B2016-TDV-06.

\end{document}